\documentclass[11pt]{article}
\usepackage[T1]{fontenc}
\usepackage{color}
\usepackage[affil-it]{authblk}
\usepackage{amssymb}
\usepackage{amsmath,mathtools}
\usepackage{amsfonts}
\usepackage{amsthm}
\usepackage{dsfont}
\usepackage{amssymb}
\usepackage{mathpazo}
\usepackage{fullpage}
\usepackage{enumerate}
\usepackage{comment}
\usepackage[cp1251]{inputenc}
\usepackage[colorlinks=true,linkcolor=blue,citecolor=blue,urlcolor=blue]{hyperref}

\usepackage{titlesec}
\titleformat{\section}[hang]{\Large\bfseries\filright}{\thesection.}{.5em}{}
\titleformat{\subsection}[hang]{\large\bfseries\filright}{}{0em}{}
\titleformat{\subsubsection}[block]{\bfseries}{}{0em}{}

\newtheorem*{theorem-non}{Theorem}

\newtheorem{theorem}{Theorem}
\numberwithin{theorem}{section}
\newtheorem{definition}[theorem]{Definition}
\newtheorem{prop}[theorem]{Proposition}
\newtheorem{lemma}[theorem]{Lemma}
\newtheorem{remark}[theorem]{Remark}
\newtheorem{corollary}[theorem]{Corollary}

\newtheorem{fact}[theorem]{Fact}

\newcommand{\tr}{\operatorname{tr}}

\title{Separable neighbourhood of identity in C$^*$-algebras}
\author [1, 2]{Mizanur Rahaman}
\author [3]{Mateusz Wasilewski}
\affil [1]{Wallenberg Centre for Quantum Technology, Chalmers University of Technology}
\affil [2]{Department of Mathematical Sciences, Chalmers University of Technology}
\affil [3] {Institute of Mathematics of the Polish Academy of Sciences}
\date{\today}

\begin{document}
\maketitle

\begin{abstract}
We study the structure of separable elements in bipartite C$^*$-algebras, focusing on the existence and size of a separable neighbourhood around the identity element. While this phenomenon is well understood in the finite-dimensional setting, its extension to general C$^*$-algebras presents additional challenges. We show that the problem of determining such a neighbourhood can be reduced to estimating the completely bounded norm of contractive positive maps. This approach allows us to characterize the size of such neighbourhoods in terms of structural properties of the algebra, notably its rank. As a consequence, we also resolve a recent conjecture of Musat and R\o{}rdam.
\end{abstract}
\section{Introduction} 

Understanding various aspects of separability and entanglement is a central theme in quantum information theory. In finite dimensional systems investigating the structure of separable states is still an active area of research (\cite{winter}, \cite{Chen}). In particular, the geometry of separable states around the identity (maximally mixed state) has been well studied and plays an important role in entanglement detection and robustness questions. However, exploring this phenomenon beyond matrix algebras to general C$^*$-algebras is both natural and technically subtle. More concretely, it is known (\cite{Gurvits-Barnum}) that in the finite dimensional setting, there is a nontrivial neighbourhood of the bipartite identity element consisting entirely of separable elements. In other words, sufficiently small perturbations of the identity cannot generate entanglement. The main question is then, what happens in an arbitrary C$^*$-algebra setting?

In this paper we undertake this task, namely, given two C$^*$-algebras $A$ and $B$, we explore the problem of finding an optimal neighbourhood of separable elements of the  bipartite identity $1_A\otimes 1_B$. By ``separable elements'' we mean the closure of the set of positive elements $x\in A\otimes_{\min} B$ that admit the decomposition $x=\sum_i a_i\otimes b_i$, where $a_i\in A_{+}, b_i\in B_{+}$. The identity element is the natural reference point of the cone of positive elements, and in matrix algberas it is the dual of the trace functional, so it makes sense to investigate this phenomenon based on the identity element. 

In matrix algebras, separability questions can often be reformulated using positive maps via the well known Horodecki criterion (\cite{MR1421501}). And using this connection, Gurvits and Burnum (\cite{Gurvits-Barnum}) obtained a separable ball of radius $1/n$ of the element $1_{M_n(\mathbb{C})}\otimes 1_{M_{n}(\mathbb{C})}$ (when the operator norm is considered). Their result also captures various other values of the  radius when one chooses various Schatten $p$-norms. 
We will mainly focus on the operator norm as for us this is the right generalization to the arbitrary C$^*$-algebra setting.  The guiding observation of our present work is that the problem of finding the size of the separable neighbourhood can be encoded in computing the completely bounded (cb for short) norm problem of positive maps. This fact was implicit in the work of Ando (\cite{Ando}) who although did not consider cb norms of positive maps, his estimates were primarily about cb norms. We spell this connection very clearly in our first result on the matrix algebra case (see Theorem \ref{thm:gammaandetacomputation}).

There has been a long effort in understanding the theory of entanglement and separability in operator algebras pioneered by St\o{}rmer (\cite{Stormer}), Arveson (\cite{arveson}), Paulsen (\cite{paulsen-book}), and other researchers (\cite{Paul-Tod-Tom, kribs}). In particular the Horodecki criterion has been generalized by St\o{}rmer (\cite{stormer-II}) to the setting where one of the algebras is nuclear and the other is an UHF algebra. Later Miller and Olkiewicz (\cite{MillerOlkiewicz}) extended this to nuclear C$^*$ algebras. Note that these generalizations were done to characterize "separable states" as opposed to ``separable elements" in a C$^*$-algebra and there is a subtle difference in these two concepts if one goes beyond the matrix algebra case. Since we are dealing with separable elements, in Section \ref{sec: Horodecki} we establish the Horodecki criterion for for separable elements in nuclear C$^*$-algebras. Having this at our disposal, we identify the problem of finding the separable neighbourhood of the bipartite identity element $1_A\otimes 1_B$ is equivalent  to the problem of computing how large the cb norm of a positive map  can be from the algebra $A$ to $B$. Now computing cb norms of completely bounded maps is a fundamental problem in operator algebras and underpins many interesting concepts of the underlying spaces (\cite{paulsen-book, Smithcb, MR723470}). In particular, the rank of a C$^*$-algebra, i.e. the largest dimension of an irreducible representation, plays a crucial role in determining how large the cb norm of a  
linear map acting on it can be. We have our main result as follows
\begin{theorem-non}
For two unital C$^*$-algebras $A, B$, 
let \[\gamma(A, B) := \sup\{r \in \mathbb{R}_{+}: \mathds{1}_A \otimes \mathds{1}_B - x \text{ is separable for all } x=x^{\ast}\text{ with }\|x\|\leqslant r\}.\]
If both have infinite rank, then $\gamma(A, B)=0$. 

\noindent In the case where at least one has a finite rank we have 

 \[\frac{1}{\gamma(A,B)} =\eta(A, B)= \min(rank(A), rank(B)) ,\]

where   $\eta(A, B) :=\sup\{\|\Phi\|_{cb}: \Phi: A\rightarrow B,  \ \text{contractive positive map}\}.$
\end{theorem-non}

Note that the statement and the proof of the  above theorem (see Theorem \ref{thm: neighbrhood-norm relation}) not only encompasses the known finite dimensional result, it also shows that the separable neighbourhood problem has no non-trivial solution when both the algebras in a bipartite system is $B(H)$, where $H$ is a separable infinite dimensional Hilbert space, as the rank is infinite. 
This fact should be compared with the result of Clifton and Halvorson (\cite{clifton}) where they showed that entangled states are dense in the trace norm topology when one subsystem of the bipartite system is infinite dimensional. 

Although our primary focus is on the separable neighbourhood problem, the methods we develop have further consequences. In particular, they enable us to resolve a conjecture raised by Magdalena Musat and Mikael R\o{}rdam in their recent work on entanglement theory in C$^*$-algebras \cite{MusatRordam}.

\section{Preliminaries}
We will study (completely) positive maps between $C^{\ast}$-algebras.
\begin{definition}
Let $\Phi: A \to B$ be a linear map between $C^{\ast}$-algebras. It is called positive if $\Phi(a^{\ast}a)\geqslant 0$ for all $a\in A$. It is called completely positive if for all $n \in \mathbb{N}$ the amplification $\Phi_n:= Id_n \otimes \Phi: M_n \otimes A \to M_n \otimes B$ is positive.   
\end{definition}
A good overview of completely positive maps can be found in \cite{paulsen-book} and \cite{BrownOzawa}. An important thing to note is that either the domain or the codomain is commutative then positive maps are automatically completely positive. There are also useful criteria for complete positivity if either the domain or the codomain is the matrix algebra.
\begin{lemma}[{\cite[Chapter 1: Proposition 5.11, Proposition 5.13]{BrownOzawa}}]\label{lem:Choimatrix}
Let $A$ be a $C^{\ast}$-algebra. The linear map $\Phi: M_n \to A$ is completely positive if and only if its \emph{Choi matrix} $Choi(A):= \sum_{i,j=1}^{n} e_{ij} \otimes \Phi(e_{ij}) \in M_{n} \otimes A$ is positive.

The linear map $\Psi: A \to M_n$ is completely positive if and only if the functional $\widehat{\Psi}: M_n\otimes A \to \mathbb{C}$ is positive, where $\widehat{\Psi}(\sum_{i,j} e_{ij}\otimes a_{ij}) = \sum_{i,j=1}^{n} \Psi(a_{ij})_{ij}$.
\end{lemma}

Another important concept will be that of a completely bounded map, i.e. a map $\Phi: A \to B$ such that the amplifications $\Phi_n: M_n(A) \to M_n(B)$ are uniformly bounded, i.e. $\sup_{n\in\mathbb{N}} \|\Phi_n\| < \infty$; this quantity will be called the cb norm of $\Phi$. Similarly to the case of positivity, bounded maps into commutative $C^{\ast}$-algebras are automatically completely bounded, but it is not the case that bounded maps with commutative domains are automatically completely bounded.

We will mostly study the maps between $C^{\ast}$-algebras by going to the induced maps between biduals, their universal enveloping von Neumann algebras. The reason for utility of this construction stems from the fact that it is easier to study maps between von Neumann algebras, and in our specific setting we will mostly have to deal with particularly nice von Neumann algebras. A lot of information on this topic can be found in \cite[Chapter III.2]{Takesaki}; a basic introduction can also be found in \cite[Chapter 1.4]{BrownOzawa}.

We will sometimes use a certain approximation property: nuclearity for $C^{\ast}$-algebras and injectivity for von Neumann algebras. An important relationship between the two is that a $C^{\ast}$-algebra is nuclear if and only if its bidual $A^{\ast\ast}$ is injective. A lot of information about these properties can be found in the book \cite{BrownOzawa}.

\section{The matrix algebra case}\label{sec:matrix}
In this section we will clarify the relationship between bounds on the cb norms of positive maps and the size of the neighbourhood consisting of separable elements. 
\begin{definition}
A bipartite element $X\in (M_n\otimes M_m)_{+}$ is called \emph{separable} if it admits a decomposition $X=\sum_i P_i\otimes Q_i$, where $0 \leq P_i \in M_n$ and $0 \leq Q_i\in M_m$. This is a closed convex subset of $M_{n} \otimes M_m$ and it will be denoted by $Sep(M_n\otimes M_{m})$.
\end{definition}

\begin{definition}\label{def:etaandgamma}
We define the following quantities:
\begin{itemize}
    \item $\eta(M_n, M_m) :=\sup\{\|\Phi\|_{cb}: \Phi: M_n\rightarrow M_m,  \ \text{contractive positive map}\}$.
\item  $\eta_{u}(M_n, M_m) :=\sup\{\|\Phi\|_{cb}: \Phi: M_n\rightarrow M_m,  \ \text{unital positive map}\}$.
\item $\gamma(M_{n}, M_{m}) := \sup\{r \in \mathbb{R}_{+}: \mathds{1}_n \otimes \mathds{1}_m - x \text{ is separable for all } x=x^{\ast}\text{ with }\|x\|\leqslant r\}$.
\end{itemize}
\end{definition}
First note that in order for the element $ \mathds{1}_n \otimes \mathds{1}_m - x$ to be even positive, we must have that $x$ is a contraction and hence the $r$ we look for in the definition of $\gamma(M_{n}, M_{m})$, must have the property that $r\leq 1$. 
The quantity $\gamma(M_n, M_m)$ tells us the radius of the largest ball around the unit in $M_n \otimes M_m$ consisting only of separable elements. To relate the parameters, we need to recall the following criterion for separability, known as the \emph{Horodeckis criterion}.
\begin{prop}[\cite{MR1421501}]\label{prop:Horodeckicrit}
Let $x \in (M_{n} \otimes M_{m})_{+}$. Then $x$ is separable if and only if $(Id\otimes \phi)(x) \geqslant 0$ for all positive maps $\phi: M_{n} \to M_{m}$. It is also equivalent to $(\psi \otimes Id)(x) \geqslant 0$ for all positive maps $\psi: M_{m} \to M_{n}$. It is sufficient to only look at unital positive maps $\phi$.
\end{prop}
\begin{remark}
Unlike the infinite dimensional situation discussed in Section \ref{sec: Horodecki}, there is a nice symmetry between the legs of the bipartite element $x$ here. The reason is that in the infinite dimensional setting we will need to add an extra assumption about one of the $C^{\ast}$-algebras, which breaks the symmetry.
\end{remark}
Before we state the relationship between $\gamma(M_{n}, M_{m})$ and $\eta(M_{n}, M_{m})$, let us note that the arguments play a symmetric role for $\gamma(M_{n}, M_{m})$, while there is a certain asymmetry in $\eta(M_{n}, M_{m})$ between the domain and the codomain; this will be reflected in the arguments. We first handle the domain.
\begin{prop}\label{prop:onesidedsepvscb}
Let $r>0$ and $n\in \mathbb{N}$. The following conditions are equivalent:
\begin{enumerate}[{\normalfont (i)}]
\item\label{item:one} $\gamma(M_n, M_{m}) \geqslant r$ for all $m\in \mathbb{N}$;
\item\label{item:two} $\eta(M_{n}, M_{m}) \leqslant \frac{1}{r}$ for all $m\in \mathbb{N}$;
\item\label{item:three} $\eta_{u}(M_{n}, M_{m}) \leqslant \frac{1}{r}$ for all $m \in \mathbb{N}$.
\end{enumerate}
\end{prop}
\begin{proof}
The proof that \eqref{item:one} implies \eqref{item:two} follows along the lines of the proof \cite[Theorem 3.7]{aubrun}, where they established that \eqref{item:one} implies \eqref{item:three}. Indeed, let $\phi: M_n \to M_{m}$ be a positive contraction and let $x \in M_{n} \otimes M_{k}$ be an element of norm $\|x\|\leqslant 1$. We want to show that $\|\phi_{k}(x)\|\leqslant \frac{1}{r}$. Consider the element $y:= \left[\begin{array}{cc} \mathds{1}_{n} \otimes \mathds{1}_{k} & rx \\ rx^{\ast} & \mathds{1}_{n} \otimes \mathds{1}_{k}\end{array}\right] \in M_{n} \otimes M_{2k}$. Define $\tilde{y}:=\left[\begin{array}{cc} 0_n\otimes 0_k & -rx \\ -rx^{\ast} & 0_{n} \otimes 0_{k}\end{array}\right] \in M_{n} \otimes M_{2k}$, then it is self-adjoint and $\|\tilde{y}||\leq r$ and hence the element $y=\left[\begin{array}{cc} \mathds{1}_{n} \otimes \mathds{1}_{k} & 0\\ 0 & \mathds{1}_{n} \otimes \mathds{1}_{k}\end{array}\right]- \tilde{y}$ is separable in $M_{n} \otimes M_{2k}$ by the assumption. Therefore by the Horodeckis criterion $\phi_{2k}(y)\geqslant 0$. But $\phi_{2k}(y) =  \left[\begin{array}{cc} \phi(\mathds{1}_{n}) \otimes \mathds{1}_{k} & r\phi_{k}(x) \\ r(\phi_{k}(x))^{\ast} & \phi(\mathds{1}_{n}) \otimes \mathds{1}_{k}\end{array}\right] $. As $\phi$ is a contraction, $\|\phi(\mathds{1}_{n})\|\leqslant 1$, which implies that $\phi(\mathds{1}_n) \leqslant \mathds{1}_m$, so $\phi_{2k}(y) \leqslant \left[\begin{array}{cc} \mathds{1}_{m} \otimes \mathds{1}_{k} & r\phi_{k}(x) \\ r(\phi_{k}(x))^{\ast} & \mathds{1}_{m} \otimes \mathds{1}_{k}\end{array}\right] $. Such a matrix is positive if and only if $\|r \phi_{k}(x)\|\leqslant 1$, so we showed that $\|\phi_{k}\| \leqslant \frac{1}{r}$, hence $\|\phi\|_{cb} \leqslant \frac{1}{r}$ as $k$ was arbitrary. 

The implication from \eqref{item:two} to \eqref{item:three} is clear, so we need to show that \eqref{item:three} implies \eqref{item:one}. Let $x=x^{\ast} \in M_{n} \otimes M_{m}$ be of norm $r$; we want to show that $\mathds{1}_{n} \otimes \mathds{1}_{m} -x$ is separable. We will employ the Horodeckis criterion (Proposition \ref{prop:Horodeckicrit}). Let $\phi: M_{n} \to M_{m}$ be a unital positive map. We have that $\|\phi_{m}(x)\| \leqslant \frac{1}{r} \|x\| = 1$. It follows that $\phi_{m}(\mathds{1}_{n}\otimes \mathds{1}_{m} - x) = \mathds{1}_{n} \otimes \mathds{1}_{m} - \phi_{m}(x) \geqslant 0$. As it happens for all unital positive maps, we conclude that $\mathds{1}_{n} \otimes \mathds{1}_{m} -x$ is separable, therefore $\gamma(M_{n}, M_{m}) \geqslant r$.
\end{proof}

We can now give the precise values of these quantities.
\begin{theorem}\label{thm:gammaandetacomputation}
We have $\gamma(M_n, M_{m}) = \frac{1}{\min(n,m)}$ and $\eta(M_n, M_m) = \eta_{u}(M_n, M_m) = \min(n,m)$.
\end{theorem}
\begin{proof}
By \cite[Corollary 8.4]{Ando}, we obtain $\gamma(M_{n}, M_{m}) \geqslant \frac{1}{\min(n,m)}$. It then follows from the previous proposition that $\eta(M_n, M_m) \leqslant n$ (as we did not discuss the dependence on the codomain in this proposition). But the inequality $\eta(M_{n}, M_m) \leqslant m$ follows from the general theory of operator spaces, as the cb norm of a map into $M_m$ is achieved at the $m$-th matricial level (see \cite[Exercise 3.10, Exercise 3.11]{paulsen-book}). On the other hand, using the transpose map, we can easily show that $\eta(M_n, M_m) \geqslant \min(n,m)$, so $\eta(M_n, M_m) = \min(n,m)$. It follows that $\gamma(M_{n}, M_{m}) = \frac{1}{\min(n,m)}$. Indeed, if $\gamma(M_{n}, M_{m}) > \frac{1}{\min(n,m)}$ then we can assume that $n \leqslant m$ by the symmetry of $\gamma$. It follows that $\gamma(M_n, M_m) > \frac{1}{n}$, which by Proposition \ref{prop:onesidedsepvscb} would imply that $\eta(M_n, M_m) < n$, which gives a contradiction.

The case of $\eta_{u}$ can be handled in the same way.
\end{proof}

Similar arguments will prove useful in the infinite dimensional case. In the remainder of this section we would like to show that one can get similar estimates for a general contraction from $M_n$ to $M_m$ without any positivity assumptions. We will rely on results about completely positive majorants from \cite{Ando} rather than invoking separable neighbourhoods; it is not possible to extend this part to the infinite dimensional setting.
\begin{definition}[{\cite[Section 5]{Ando}}]
  For a linear map $\Psi: M_m \to M_n$ a pair of  maps $\Phi_1, \Phi_2: M_m \to M_n$ is called a completely positive majorizing pair if the block matrix 
\[
\left[\begin{array}{cc} A & B \\ B^{\ast} & C \end{array}\right]
\]
is positive, where $B$ is the Choi matrix of $\Psi$ and $A$ and $C$ are the respective Choi matrices of $\Phi_1$ and $\Phi_2$.  
\end{definition}

\begin{prop}
Let $\Psi: M_n \to M_m$ be a linear map. Then $\|\Psi\|_{cb} \leqslant \min(m,n) \|\Psi\|$.
\end{prop}
\begin{proof}
By \cite[Corollary 6.7]{Ando} $\Psi$ admits a majorizing pair $(\Phi_1, \Phi_2)$ such that $\|\Phi_j\| \leqslant \min(n,m) \|\Psi\|$ for $j \in \{1,2\}$. Hence if $A, C$ are the Choi matrices of $\Phi_1, \Phi_2$ (respectively) and $B$ denotes the Choi matrix of $\Psi$, then we have that the block matrix $
\left[\begin{array}{cc} A & B \\ B^{\ast} & C \end{array}\right]
$ is positive.

We will check that for each $k \in \mathbb{N}$ the pair $(Id_{k}\otimes \Phi_{1}, Id_{k}\otimes \Phi_{2})$ is a majorizing pair for $Id_k\otimes \Psi$. These are maps from $ M_{k} \otimes M_{n} \simeq M_{nk}$ to $M_{k} \otimes M_{m} \simeq M_{mk}$. To compute the Choi matrix of such a map, we will use the matrix units $e_{ijrs} := e_{ij} \otimes f_{rs} $, where the $e_{ij}$ are the standard matrix units in $M_{k}$ and $f_{rs}$ are the standard matrix units in $M_{n}$. Then the Choi matrix of $Id_k \otimes \Psi$ is equal to $\sum_{i,j,r,s} (Id_k \otimes \Psi)(e_{ijrs}) \otimes e_{ijrs}$. We can compute it further to obtain $\sum_{i,j,r,s} (e_{i,j}\otimes \Psi(f_{ij})) \otimes e_{ij} \otimes f_{rs}$; we can calculate the Choi matrices of $Id_k\otimes \Phi_{1}$ and $Id_k\otimes \Phi_{2}$ in the same way. Note that there is a $\ast$-isomorphism $M_{m} \otimes M_{k} \otimes M_{n} \otimes M_{k} \simeq M_{m} \otimes M_{n} \otimes M_{k} \otimes M_{k}$ flipping the second and the third legs of the tensor product. For checking the majorizing condition, we need to check positivity of a certain $2 \times 2$ matrix with entries coming from this algebra, which is definitely a property preserved under the $\ast$-isomorphism. This isomorphism maps the Choi matrix of $Id_k\otimes \Psi$ to $P_k\otimes B$, where $B$ is the Choi matrix of $\Psi$ and $P_k$ is the Choi matrix of $Id_{k}$; of course the same works for $\Phi_1$ and $\Phi_2$. It follows that we need to show positivity of the following matrix 
\[
\left[\begin{array}{cc} P_{k} \otimes A & P_{k} \otimes B  \\ P_{k} \otimes B^{\ast} & P_{k} \otimes C \end{array}\right],
\]
which is just the tensor product of the original matrix $
\left[\begin{array}{cc} A & B \\ B^{\ast} & C \end{array}\right]
$ tensored with $P_k$, so it is positive.

It now follows from \cite[Theorem 5.2]{Ando} that $\|Id_k\otimes \Psi\| \leqslant \sqrt{\|Id_k\otimes \Phi_1\|\cdot \|Id_k \otimes \Phi_2\|}$, hence $\|\Psi\|_{cb} \leqslant \sqrt{\|\Phi_1\|_{cb}\cdot \|\Phi_2\|_{cb}}$. But $\Phi_j$ is completely positive, so $\|\Phi_j\|_{cb} = \|\Phi_j\|$. We conclude that 
\[
\|\Psi\|_{cb} \leqslant \sqrt{\|\Phi_1\|\cdot \|\Phi_2\|} \leqslant \min(n,m) \|\Psi\|.
\]
\end{proof}
We can now state the version of \cite[Theorem 3.7]{aubrun} for general linear maps.
\begin{corollary}
Let $\Phi:M_n \to V$ be a linear map into an operator space $V$. Then $\|\Phi\|_{cb}\leqslant n \|\Phi\|$.
\end{corollary}
\begin{proof}
Embed $V$ completely isometrically into some $B(H)$ and consider the composed map $\widetilde{\Phi}: M_n \to B(H)$. It suffices to show that $\|\widetilde{\Phi}\|_{cb} \leqslant n \|\widetilde{\Phi}\|$, as one has $\|\widetilde{\Phi}\|_{cb} = \|\Phi\|_{cb}$ and $\|\widetilde{\Phi}\| = \|\Phi\|$. Let $P_{k}: H \to H$ be an orthogonal projection of rank $k$ and assume that the sequence $P_{k}$ increases to the identity. Then we can approximate $\widetilde{\Phi}$ (and its amplifications) by $P_{k}\widetilde{\Phi} P_{k}$ (and its amplifications). But $P_{k} B(H) P_{k} \simeq M_{k}$, so by the previous proposition we get $\|P_{k} \widetilde{\Phi}P_{k}\|_{cb} \leqslant n \|P_{k}\widetilde{\Phi} P_{k}\|$. Since the domain is finite dimensional, this gives a convergence in norm, therefore $\|\Phi\|_{cb}= \|\widetilde{\Phi}\|_{cb} \leqslant n\|\widetilde{\Phi}\| = n \|\Phi\|$.
\end{proof}

\section{Horodeckis criterion for von Neumann algebras}{\label{sec: Horodecki}}


An important generalization of the Horodeckis criterion appeared in \cite{MillerOlkiewicz}. However their results concern separable states and we work with separable elements in the $C^{\ast}$-algebras; unlike the finite dimensional case, you cannot automatically translate between the two frameworks. Fortunately their proof can be adapted to our setting and we will provide some details for the convenience of the reader. Before that, let us formally define our sets of separable elements, as we will have to distinguish the cases of $C^{\ast}$-algebras and von Neumann algebras.
\begin{definition} 
If $A$ and $B$ are $C^{\ast}$-algebras. We denote by $Sep(A\otimes B)$ the set of separable elements of $A\otimes B$, i.e. the closure of the finite sums of the form $\sum_{i} a_{i} \otimes b_{i}$, where the $a_{i}$ and the $b_{i}$ are positive. 

For von Neumann algebras $M$ and $N$ we will use the notation $Sep_{\sigma}(M\overline{\otimes} N)$ to define an analogous set, this time taking the closure in the weak$^{\ast}$-topology; elements of this set will be called separable, other positive elements will be called entangled. 
\end{definition}
\begin{remark}
In the definition of $Sep(A\otimes B)$ we might as well take the closure in the weak topology, because the set is convex and the closure of a convex set is the same with respect to the weak and the norm topologies, as a consequence of the Hahn-Banach separation theorem. This will be useful later, because the weak topology on $A$ is equal to the topology induced from the weak$^{\ast}$ topology of $A^{\ast\ast}$.
\end{remark}
The main goal of this section is the following theorem, analogous to the main result of \cite{MillerOlkiewicz}.
\begin{theorem}\label{thm:Horodecki}
Let $x \in M\overline{\otimes} N$, where $M$ is an injective von Neumann algebra. Then $x\in Sep_{\sigma}(M\overline{\otimes} N)$ if and only if $(Id\otimes \phi)(x) \in M\overline{\otimes} M$ is positive for all positive, normal, finite rank maps $\phi: N \to M$.
\end{theorem}
Note that one direction is clear, so we only have to show that if applying positive maps to one leg of the tensor preserves positivity then the element has to be separable. The crucial first step, just like in \cite{MillerOlkiewicz}, is the case $M=M_d$ and then we will use structural results for von Neumann algebras to conclude the general case.
\begin{lemma}\label{lem:findimhorodecki}
Let $x \in M_d \otimes N$. Then $x\in Sep_{\sigma}(M_d\otimes N)$ if and only if $(Id\otimes \phi)(x) \in M_d\otimes M_d$ is positive for all normal positive maps $\phi: N \to M_d$. 
\end{lemma}
\begin{proof}
We will use the correspondence between maps $\phi: N \to M_d$ and functionals $\widehat{\phi}$ on $M_d \otimes N$, given by $\widehat{\phi}(\sum_{i} a_{i} \otimes b_{i}) := Tr(\sum_{i} a_{i}^{T} \phi(b_{i}))$ (see Lemma \ref{lem:Choimatrix}). This replaces the use of the Choi-Jamio{\l}kowski isomorphism in our dual setting. It is clear that $\widehat{\phi}$ is positive on separable elements if and only if $\phi$ is a positive map. Because the set of separable elements is a weak$^{\ast}$ closed cone, we can separate it by a functional from any entangled element. It means that an element $x\in M_{d}\otimes N$ is separable if and only if $\widehat{\phi}(x)\geqslant 0$ for all positive maps $\phi: N \to M_d$. 

Let $P= \sum_{k,l} e_{kl}\otimes e_{kl} \in M_{d} \otimes M_d$; it is a positive operator, as it is a multiple of a (rank one) projection. Let $x = \sum_{i} a_{i} \otimes b_{i}$ and suppose that $(Id \otimes \phi)(x) \geqslant 0$. Then $(Tr\otimes Tr)(P (Id \otimes \phi)(x))\geqslant 0$. We can compute this value explicitly as
\[
\sum_{i,k,l} Tr(e_{kl} a_{i})Tr(e_{kl}\phi(b_i)) = \sum_{i,k,l} (a_{i}^{T})_{kl} (\phi(b_i))_{lk} = Tr(\sum_{i} a_{i}^{T} \phi(b_{i})) = \widehat{\phi}(x) \geqslant 0.
\]
It follows from our previous discussion that $x$ has to be separable.
\end{proof}
\begin{lemma}\label{lem:separablesplit}
Let $M := \bigoplus_{i} M_{i}$. Then $x \in M \overline{\otimes} N$ is separable if and only if its components $x_{i} \in M_{i} \overline{\otimes} N$ are separable.
\end{lemma}
\begin{proof}
If the components $x_i$ are separable then $x$ is separable. On the other hand, if $x$ is separable, then $x_{i} = z_{i} x$ is separable, where $z_i$ is the central projection corresponding to the summand $M_i$.
\end{proof}
\begin{corollary}\label{cor:findimhorodecki}
Theorem \ref{thm:Horodecki} holds for finite dimensional von Neumann algebras. Moreover, it is sufficient to treat each type separately.
\end{corollary}
\begin{proof}
Any finite dimensional von Neumann algebra is a direct sum of matrix algebras, so Lemma \ref{lem:findimhorodecki} applies. Any von Neumann algebra splits into a direct sum of von Neumann algebras of a fixed type, so can treat each type separately. 
\end{proof}
The next steps are standard reductions used in the von Neumann algebra theory.
\begin{lemma}
Let $M$ be a finite, injective, $\sigma$-finite von Neumann algebra. Then $x\in Sep_{\sigma}(M\overline{\otimes} N)$ if and only if $(Id\otimes \phi)(x) \in M\overline{\otimes} M$ is positive for all positive, normal, finite rank maps $\phi: N \to M$.
\end{lemma}
\begin{proof}
Since $M$ is finite and $\sigma$-finite, it admits a faithful normal tracial state $\tau$. It follows that any von Neumann subalgebra $M' \subset M$ admits a normal faithful conditional expectation $\mathbb{E}_{M'}: M \to M'$. As $M$ is injective, it is an increasing union of finite dimensional subalgebras $M_{k}$ with the corresponding conditional expectations $\mathbb{E}_{k}$.

Let now $x \in M \overline{\otimes} N$ satisfy that $(Id\otimes \phi)(x) \in M\overline{\otimes} M$ is positive for all positive, normal, finite rank maps $\phi: N \to M$. Now let's define $x_k:=(\mathbb{E}_{k}\otimes Id)(x)$, then $x_k\rightarrow x$ in the weak$^*$ sense as $\mathbb{E}_{k}$ converges in the point-weak$^{\ast}$ topology to identity. Now note that 
\[(Id\otimes \phi)(x_k)=(Id\otimes \phi)(\mathbb{E}_{k}\otimes Id)(x)=(\mathbb{E}_{k}\otimes Id)(Id\otimes \phi)(x)\geq0,\]
where we have used that $(Id\otimes \phi)(x)\geq 0$ and the fact that $E_k$ is a completely positive map. 
 It follows from Lemma \ref{lem:findimhorodecki} that $x_k$ is separable for all $k \in \mathbb{N}$. So $x$ is separable as well as the set of separable elements is weak$^*$ closed. 
\end{proof}
\begin{lemma}\label{lem:semifinite}
Theorem \ref{thm:Horodecki} holds if $M$ is semifinite, injective and $\sigma$-finite.
\end{lemma}
\begin{proof}
One can consider an increasing family of projections $p_{k}$ converging to identity such that the corner $p_{k} M p_{k}$ is finite and $\sigma$-finite. In $p_{k}M p_{k}$ one can use the previous lemma and then use the convergence of $p_{k}$ to show that the same holds in $M$.    
\end{proof}

\begin{lemma}
Theorem \ref{thm:Horodecki} holds if $M$ is an injective $\sigma$-finite von Neumann algebra of type III.
\end{lemma}
\begin{proof}
One can use the same proof as in \cite[Lemma 3.7]{MillerOlkiewicz}, using the fact that the continuous core is semifinite.
\end{proof}
\begin{proof}[Proof of Theorem \ref{thm:Horodecki}]
It follows from all the lemmas above that the result holds for $\sigma$-finite von Neumann algebras, as we handled the type III case and the semifinite case, and every von Neumann algebras splits into a direct sum of those two cases. The last step is to get rid of the $\sigma$-finite assumption. This can be done similarly to the Lemma \ref{lem:semifinite}, because in an arbitrary von Neumann algebra $M$ one can find an increasing family of $\sigma$-finite projections converging to identity.
\end{proof}

\begin{corollary}\label{cor:horodeckicstar}
Let $A$ and $B$ be $C^{\ast}$-algebras and assume that $A$ is nuclear. Then $x\in Sep(A\otimes B)$ if and only if $(Id \otimes \phi)(x)\geqslant 0$ for all finite rank, positive maps $\phi: B \to A$.    
\end{corollary}
\begin{proof}
One can prove it in exactly the same way as \cite[Theorem 4.1]{MillerOlkiewicz} (see also the proof of Proposition \ref{prop:cbnormbidual}, and Proposition \ref{prop:separablebidual}). The basic idea is that if $x\in Sep(A\otimes B)$, then $x$ is separable in $A^{**}\otimes B^{**}$. Since $A$ is nuclear, $M:=A^{**}$ is injective. Now we can use Theorem \ref{thm:Horodecki} and the fact that any positive map $\Phi: B\rightarrow A$ induces a normal positive map on the biduals.
\end{proof}
We will also show that it suffices to work with unital positive maps.
\begin{lemma}
Let $x \in M\overline{\otimes} N$, where $M$ is an injective von Neumann algebra. If $(Id\otimes \phi)(x)\geqslant 0$ for all normal, unital, finite rank, positive maps $\phi: N\to M$ then $x$ is separable.
\end{lemma}
\begin{proof}
We will first show that it suffices to consider maps $\phi$ for which $\phi(\mathds{1}_{N})$ is invertible. Indeed, let $\varphi: N \to \mathbb{C}$ be an arbitrary state. Then we can define $\phi_{\varepsilon}(n):= \phi(n) +\varepsilon \varphi(n)\mathds{1}_{M}$. It follows that $\phi_{\varepsilon}(\mathds{1}_N) \geqslant \varepsilon \mathds{1}_{M}$, so it is invertible. By convergence $\lim_{\varepsilon \to 0} \phi_{\varepsilon} = \phi$ it follows that if $(Id\otimes \phi_{\varepsilon})(x)\geqslant 0$ for all $\varepsilon > 0$, then $(Id\otimes \phi)(x)\geqslant 0$.

If $\phi(\mathds{1}_{N})$ is invertible then one can consider the unital positive map $\widetilde{\phi}(n) := \phi(\mathds{1}_{N})^{-\frac{1}{2}}\phi(n) \phi(\mathds{1}_{N})^{-\frac{1}{2}}$. Clearly $(Id\otimes \phi)(x)\geqslant 0$ if and only if $(Id\otimes \widetilde{\phi})(x)\geqslant 0$ as $\widetilde{\phi}$ and $\phi$ differ by a completely positive map with a completely positive inverse.
\end{proof}

\section{The C$^*$-algebra case}
In this section we will explore how the results from Section \ref{sec:matrix} can be extended to the infinite dimensional setting, inspired by the recent paper \cite{MusatRordam}. We will start with the analogue of Definition \ref{def:etaandgamma}, where we have to distinguish between the case of $C^{\ast}$-algebras and von Neumann algebras.
\begin{definition}
Let $A$ and $B$ be $C^{\ast}$-algebras. We define  
\begin{itemize}
    \item $\eta(A, B) :=\sup\{\|\Phi\|_{cb}: \Phi: A\rightarrow B,  \ \text{contractive positive map}\}$.
\item  $\eta_{u}(A, B) :=\sup\{\|\Phi\|_{cb}: \Phi: A\rightarrow B,  \ \text{unital positive map}\}$.
\item $\gamma(A, B) := \sup\{r \in \mathbb{R}_{+}: \mathds{1}_A \otimes \mathds{1}_B - x \text{ is separable for all } x=x^{\ast}\text{ with }\|x\|\leqslant r\}$.
\end{itemize}
For $M$ and $N$ von Neumann algebras we define
\begin{itemize}
    \item $\eta^{\sigma}(M, N) :=\sup\{\|\Phi\|_{cb}: \Phi: M\rightarrow N,  \ \text{normal, contractive, positive map}\}$.
\item  $\eta_{u}^{\sigma}(M, N) :=\sup\{\|\Phi\|_{cb}: \Phi: M\rightarrow N,  \ \text{normal, unital, positive map}\}$.
\item $\gamma^{\sigma}(M, N) := \sup\{r \in \mathbb{R}_{+}: \mathds{1}_M \otimes \mathds{1}_N - x \text{ is separable for all } x=x^{\ast}\text{ with }\|x\|\leqslant r\}$.
\end{itemize}

\end{definition}

We will start our analysis by establishing a correspondence between the $\eta$ and $\gamma$ parameters of $C^{\ast}$-algebras and their biduals.
\begin{prop}\label{prop:separablebidual}
We have $\gamma(A,B) = \gamma^{\sigma}(A^{\ast\ast}, B^{\ast\ast})$.
\end{prop}
\begin{proof}
Suppose that $\gamma(A,B)=r$, i.e. for any Hermitian element $H \in A\otimes B$ of norm at most $r$ the element $\mathds{1}_{A} \otimes \mathds{1}_{B}- H$ is separable. Let now $K=K^{\ast} \in A^{\ast\ast}\overline{\otimes} B^{\ast\ast}$ be of norm at most $r$. By Kaplansky's density theorem, we can approximate it in the strong operator topology by Hermitian elements $K_{i}$ from $A\otimes B$ of norm at most $r$. By our assumption the elements $\mathds{1}_{A} \otimes \mathds{1}_{B} - K_{i}$ are separable in $A\otimes B$. Since they converge strongly, in particular weak$^{\ast}$, to $\mathds{1}_{A}\otimes \mathds{1}_{B} - K \in A^{\ast\ast}\overline{\otimes} B^{\ast\ast}$, these elements are separable as well. It follows that $\gamma(A^{\ast\ast}, B^{\ast\ast}) \geqslant r = \gamma(A,B)$.

Now assume that $\gamma(A^{\ast\ast}, B^{\ast\ast}) = r$ and let $H=H^{\ast} \in A\otimes B$ be a Hermitian element of norm at most $r$. It follows that the element $\mathds{1}_{A} \otimes \mathds{1}_{B}- H$ is separable in $A^{\ast\ast}\overline{\otimes} B^{\ast\ast}$. Therefore $\mathds{1}_{A}\otimes \mathds{1}_{B} - H$ is a weak$^{\ast}$-limit of finite sums of the form $\sum_{i} a_{i} \otimes b_{i}$, where $0\leqslant a_{i} \in A^{\ast\ast}$ and $0\leqslant b_{i} \in B^{\ast\ast}$. We can use Kaplansky's theorem to approximate each $a_{i}$ and $b_{i}$ by positive elements in $A$ and $B$, respectively. Therefore $\mathds{1}_{A} \otimes \mathds{1}_{B} - H$ can be approximated in the weak$^{\ast}$ topology of $A^{\ast\ast}\overline{\otimes} B^{\ast\ast}$ by finite sums of the form $\sum_{i} P_{i} \otimes Q_{i}$, with $0\leqslant P_i \in A$ and $0 \leqslant Q_i \in B$. But this means that these finite sums converge to $\mathds{1}_{A} \otimes \mathds{1}_{B} - H$ in the weak topology of $A\otimes B$. As we already mentioned, the cone of separable elements is closed and convex, therefore weakly closed, hence we conclude that the elements $\mathds{1}_{A} \otimes \mathds{1}_{B} - H$ are separable. As a consequence $\gamma(A,B) \geqslant r = \gamma(A^{\ast\ast}, B^{\ast\ast})$, which finishes the proof.
\end{proof}
Under extra assumptions, we will now obtain a similar result for $\eta(A,B)$.
\begin{prop}\label{prop:cbnormbidual}
Assume that $B$ is a nuclear $C^{\ast}$-algebra. Then $\eta(A,B) = \eta^{\sigma}(A^{\ast\ast}, B^{\ast\ast})$.    
\end{prop}
\begin{proof}
Let $\phi: A \to B$ be a positive map. We can always consider the bidual map $\phi^{\ast\ast}: A^{\ast\ast} \to B^{\ast\ast}$, which has the same norm and cb norm as $\phi$. It follows that $\eta(A,B)\leqslant \eta^{\sigma}(A^{\ast\ast}, B^{\ast\ast})$.

To go the other way, the issue is that not all normal maps from $A^{\ast\ast}$ to $B^{\ast\ast}$ arise from maps from $A$ to $B$. Let $\psi: A^{\ast\ast} \to B^{\ast\ast}$ be a normal positive map. We can consider the restriction $\psi_{|A}$, which will have the same norm and cb norm as $\psi$, by using Kaplansky's density theorem. Let us verify it for the norm; the proof for the cb norm will be the same as we can run the same argument for amplifications. Indeed, suppose that for all contractions $x\in A$ we have $\|\psi(x)\|\leqslant C$. For a contraction $x \in A^{\ast\ast}$ we can find a net of contractions $x_i \in A$ converging strongly to $x$, in particular in the weak$^{\ast}$-topology. Therefore $\lim_{i} \psi(x_i) = \psi(x)$ in the weak$^{\ast}$ topology. As we have $\|\psi(x_i)\|\leqslant C$ and the balls are weak$^{\ast}$ closed, it follows that $\|\psi(x)\|\leqslant C$, therefore $\|\psi\| \leqslant \|\psi_{|A|}\|$, and the other inequality is clear.

Now we have the map $\psi_{|A}: A \to B^{\ast\ast}$ and we have to construct maps going into $B$. Here we will use nuclearity. Indeed, nuclearity of $B$ forces $B^{\ast\ast}$ to be injective, therefore we can approximate the identity map by finite rank, normal completely positive maps. We will also use the fact that nuclearity implies local reflexivity (see \cite[Chapter 9: Theorem 3.1]{BrownOzawa}). Recall that a $C^{\ast}$-algebra $B$ is locally reflexive if for any finite dimesional operator subsystem $E \subset B^{\ast\ast}$ there is a net of contractive, completely positive maps $\varphi_{i}: E \to B$ that converges to $Id_{E}$ in the point-weak$^{\ast}$ topology.

Now let $v: B^{\ast\ast} \to B^{\ast\ast}$ be a finite rank, normal completely positive map. Then its image is a finite dimensional self-adjoint subspace $B^{\ast\ast}$, hence an operator system (possibly after adding the unit). By local reflexivity, we can find a contractive, completely positive map from $Im(v)$ to $B$, which approximates the identity map on $Im(v)$ in the point-weak$^{\ast}$ topology. It follows that we can construct a net of finite rank, normal, completely positive contractions $E_{i}: B^{\ast\ast} \to B \subset B^{\ast\ast}$ which converge to the identity map in the point-weak$^{\ast}$ topology. We have $\|E_{i} \circ \psi_{|A}\|_{cb} \leqslant \eta(A,B) \|E_{i}\circ \psi_{|A}\| \leqslant \eta(A,B) \|\psi_{|A}\| = \eta(A,B) \|\psi\|$. For each amplification we can prove that $\|\psi_n\| = \|(\psi_{|A})_{n}\| \leqslant \limsup_{i} \|(E_{i}\circ \psi_{|A})_{n}\|\leqslant \eta(A,B) \|\psi\| $, from which it follows that $\|\psi\|_{cb} \leqslant \eta(A,B) \|\psi\|$, hence $\eta(A^{\ast\ast}, B^{\ast\ast}) \leqslant \eta(A,B)$. 
\end{proof}
\begin{remark}
An alternative argument to finding maps from $B^{\ast\ast}$ into $B$ can be given using Choi matrices. Indeed, injectivity of $B^{\ast\ast}$ is also equivalent to the existence of approximate factorizations $\phi_{\lambda}: B^{\ast\ast} \to M_{n_{\lambda}}$ and $\psi_{\lambda}: M_{n_{\lambda}}\to B^{\ast\ast}$ such that $\psi_{\lambda} \circ \phi_{\lambda}$ converges to $Id_{B^{\ast\ast}}$ in the point-weak${^\ast}$ topology. As the domain of $\psi_{\lambda}$ is a matrix algebra, we can use Lemma \ref{lem:Choimatrix} to view it as an element of $M_{n_{\lambda}}(B^{\ast\ast})$. By Kaplansky's density theorem, this element can be approximated by positive elements in $M_{n_{\lambda}}(B)$, i.e. completely positive maps from $M_{n_{\lambda}}$ into $B$. By composing with $\phi_{\lambda}$, we obtain maps from $B^{\ast\ast}$ to $B \subset B^{\ast\ast}$, which converge to the identity on $B^{\ast\ast}$ in the point-weak$^{\ast}$ topology.
\end{remark}
\subsection{Relation with the rank of subhomogeneous algebras}
The authors of \cite{MusatRordam} noted the relationship between the problems we are studying here and subhomogeneous $C^{\ast}$-algebras; it will turn out that this is exactly the class of $C^{\ast}$-algebras, where something nontrivial happens.
\begin{definition}
Let $A$ be a $C^{\ast}$-algebra. We define $rank(A) \in \mathbb{N} \cup \{\infty\}$ to be the supremum of the set of  $n \in \mathbb{N} \cup \{\infty\}$ such that $A$ has an irreducible representation of dimension $n$. A $C^{\ast}$-algebra is called \emph{subhomogeneous} if $rank(A) < \infty$.
\end{definition}
It is very easy to describe the biduals of subhomogeneous $C^{\ast}$-algebras and this will allow us to use von Neumann algebraic techniques. 
\begin{fact}
A $C^{\ast}$-algebra $A$ is subhomogeneous of rank $n$ if and only if its bidual is of the form $\bigoplus_{k=1}^{r} M_{n_{k}} \otimes L^{\infty}(X_{k})$, where $n_{r} = n$.
\end{fact}
\begin{proof}
This follows from the fact that if the bidual $A^{\ast\ast}$ has only finite dimensional irreducible representations, it has to be of type I, therefore the decomposition follows from general properties of type I von Neumann algebras. See also \cite[Proposition IV.1.4.6]{Blackadar} or \cite[Chapter 2: Proposition 7.7]{BrownOzawa}.
\end{proof}

To relate subhomogeneous $C^{\ast}$-algebras to matrix algebras, we will use results by Smith (see \cite[Lemma 2.7]{Smithcb}). Alternatively, one can use Glimm's lemma, as it was done in \cite[Lemma 3.9]{MusatRordam}).
\begin{lemma}[{\cite[Lemma 2.7]{Smithcb}}]\label{lem:smithfactorization}
Let $A$ be a $C^{\ast}$-algebra admitting an irreducible representation of dimension at least $n$. Then for any $\varepsilon > 0$ there are unital completely positive maps $\rho: M_n \to A$ to and $\sigma: A \to M_n$ such that $\|\sigma \circ \rho - Id\|_{cb} < \varepsilon$
\end{lemma}
\begin{corollary}\label{cor:etaestimate}
We have $\eta(A,B) \geqslant \eta(A, M_{n})$ for any $n \leqslant rank(B)$.    
\end{corollary}
\begin{proof}
By the previous lemma for any $n \in \mathbb{N}$ we can find unital completely positive, almost completely isometric embeddings of $M_n$ into $B$, from which the result follows.
\end{proof}

We will first extend the results about separable neighbourhoods to tensor products with commutative algebras.
\begin{lemma}\label{lem:separable}
Let $T=T^{\ast} \in (M_n \otimes L^{\infty}(X))\otimes M_d$ satisfy $\|T\|\leqslant \frac{1}{min(d,n)}$. Then $(1_{n} \otimes \mathds{1}_{X})\otimes 1_{d} - T$ is a separable element.

Consequently, for any C$^*$-algebra $B$ and any positive contraction $\Phi: M_n \otimes L^{\infty}(X)\rightarrow B$ we have  $||\Phi||_{cb}\leq n$. In particular, it holds for unital positive maps.

Moreover, all the above assertions hold if we replace $L^{\infty}(X)$ by $C(X)$.
\end{lemma}
\begin{proof}
We interpret the tensor product $(M_n \otimes L^{\infty}(X))\otimes M_d$ as $M_{n}\otimes M_{d}$-valued measurable functions on $X$. By \cite[Corrolary 8.4]{Ando} the function corresponding to the element $(1_{n} \otimes \mathds{1}_{X})\otimes 1_{d} - T$ takes values in separable elements pointwise almost everywhere. Let $f$ be such a function. Then we can approximate it by simple functions, i.e. finite valued ones. Since the cone of separable elements is closed and convex, the values of these simple functions  themselves are separable elements, so can be approximated by finite sums, from which it follows that $f$ can be approximated by an expression of the form $\sum_{i} \mathds{1}_{A_{i}} \otimes P_{ij} \otimes Q_{ij}$, where the $P_{ij}$ and the $Q_{ij}$ are positive. This expression is manifestly separable; it follows that $f$ itself is separable. The case for $C(X)$ follows similarly and we see the algebra $(M_n \otimes C(X))\otimes M_d$ as $M_{n}\otimes M_{d}$-valued continuous functions on $X$. Ando's result gives us a pointwise separable function $f: X\rightarrow M_{n}\otimes M_{d}$ corresponding to the element $(1_{n} \otimes \mathds{1}_{X})\otimes 1_{d} - T$. But as the set of separable elements is closed, we get that $f$ itself is separable; in this case we use partitions of unity rather than approximations by simple functions.

The proof of the second assertion follows similarly to Proposition \ref{prop:onesidedsepvscb}. Indeed, let $r \in \mathbb{N}$ and let $x \in (M_{n} \otimes L^{\infty}(X))\otimes M_{r}$ be of norm one. We can form the matrix 
\[
y:= \left[\begin{array}{cc} (\mathds{1}_{n} \otimes \mathds{1}_{X})\otimes \mathds{1}_{r} & \frac{1}{n} x \\ \frac{1}{n}x^{\ast} & (\mathds{1}_{n} \otimes \mathds{1}_{X})\otimes \mathds{1}_{r}\end{array}\right] \in (M_{n} \otimes L^{\infty}(X))\otimes M_{2r},
\]whose distance from the unit in this algebra is not greater than $\frac{1}{n}$. It follows that $y$ is separable, so we can finish the proof exactly as in the proof of Proposition \ref{prop:onesidedsepvscb}. The case for $C(X)$ follows similarly.
\end{proof}
\begin{corollary}
Let $A$ be a subhomogeneous $C^{\ast}$-algebra. Then $\gamma(A, M_{d}) \geqslant \frac{1}{rank(A)}$.
\end{corollary}
\begin{proof}
We first use Proposition \ref{prop:separablebidual} to see that $\gamma(A, M_d) = \gamma^{\sigma}(A^{\ast\ast}, M_d)$. As $A$ is subhomogeneous, $A^{\ast\ast}$ splits as a direct sum $\sum_{k=1}^{r} M_{n_{k}} \otimes L^{\infty}(X_k)$, where $n_{r} = rank(A)$. Since separable elements in direct sums are just direct sums of separable elements, we can focus on a single summand and then the result follows from the previous lemma.
\end{proof}


We are now able to prove our first estimate on $\eta(A,B)$.
\begin{lemma}
We have $\sup_{B} \eta(A,B) = rank(A)$.
\end{lemma}
\begin{proof}
Note first that it suffices to show that $\sup_{B} \eta(A,B) \geqslant rank(A)$ snd $\eta^{\sigma}(A^{\ast\ast}, B^{\ast\ast}) \leqslant rank(A)$. To this end it will be enough to take the supremum over matrix algebras. Indeed, we have $\eta(A, M_d) = \eta^{\sigma}(A^{\ast\ast}, M_d)$. Since $A^{\ast\ast}$ contains $M_{rank(A)}$ as a direct summand (where by $M_{\infty}$ we mean $B(\ell^2)$), we can use the transpose map to see that $\eta(A, M_d) = \eta^{\sigma}(A^{\ast\ast}, M_d) \geqslant \min(rank(A), d)$. Letting $d$ go to infinity proves the first part of the claim.

For the second part, if $rank(A)=\infty$, then nothing to prove. Otherwise we use Lemma \ref{lem:separable}. Indeed, it follows from the lemma that any positive contraction from $  M_n \otimes L^{\infty}(X)$ into any $C^{\ast}$-algebra has cb norm bounded by $n$. If $n=rank(A)<\infty$, then  $A^{\ast\ast}$ contains a summand of $M_n \otimes L^{\infty}(X)$, so for any positive map $\phi: A^{\ast\ast} \to B^{\ast\ast}$ we have $||\Phi||_{cb}\leq n$ . This is enough to conclude that $\eta(A,B) \leqslant \eta^{\sigma}(A^{\ast\ast},B^{\ast\ast})\leqslant rank(A)$.
\end{proof}

\begin{prop}
We have $\eta(A,B) = \min(rank(A), rank(B))$.
\end{prop}
\begin{proof}
From the previous lemma we have $\eta(A,B) \leqslant rank(A)$ and $\eta(A,B) \leqslant rank(B)$ is well-known, but we will sketch the proof. The statement is vacuous if $rank(B) = \infty$ and if $B$ is subhomogeneous, \cite[Theorem 2.1]{MR723470} shows that all bounded maps into $B$ are completely bounded. As $B$ embeds into $\widetilde{B}:=M_{rank(B)} \otimes C(K)$, it suffices to bound the cb norm of a positive contraction into $\widetilde{B}$. But $\widetilde{B}$ consists of matrix-valued functions, so by composing with point evaluations, we may reduce to the case $M_{rank(B)}$. It was proved in \cite[Theorem 2.10]{Smithcb} that in this case $\|\phi\|_{cb} = \|\phi_{rank(B)}\| \leqslant rank(B) \|\phi\|$; the last estimate is well-known, see \cite[Exercise 3.10, Exercise 3.11]{paulsen-book}. It folows that $\eta(A,B) \leqslant \min(rank(A), rank(B))$.

To get the other estimate note that by Corollary \ref{cor:etaestimate} we have $\eta(A,B) \geqslant \eta(A, M_n)$ for any natural number $n\leqslant rank(B)$. As $M_n$ is nuclear, we have $\eta(A, M_n) = \eta^{\sigma}(A^{\ast\ast}, M_n)$. We can now use the transpose map to conclude that $\eta(A,B) \geqslant \min(rank(A), n)$ for any natural number $n\leqslant rank(B)$, i.e. $\eta(A,B) \geqslant \min(rank(A), rank(B))$.

\end{proof}

Now we can state and prove the main result. 
\begin{theorem}\label{thm: neighbrhood-norm relation}
We have $\eta(A,B) = \min(rank(A), rank(B)) = \frac{1}{\gamma(A,B)}$.
\end{theorem}
\begin{proof}
We know from the previous proposition that $\eta(A,B) = \min(rank(A), rank(B))$, so we only need to address the case of $\gamma(A,B)$. If neither $A$ nor $B$ is subhomogeneous, we can use $\gamma(A,B) = \gamma^{\sigma}(A^{\ast\ast}, B^{\ast\ast})$ and the fact that both $A^{\ast\ast}$ and $B^{\ast\ast}$ contain a direct summand isomorphic to $B(\ell^2)$ to conclude that $\gamma(A,B) = 0$; it follows from Theorem \ref{thm:gammaandetacomputation}.

Now we can assume that either $A$ or $B$ is subhomogeneous. As $\gamma(A,B)$ is symmetric in the variables, we can assume that $A$ is subhomogeneous without loss of generality. We will show that $\gamma(A,B) = \frac{1}{\min(rank(A), rank(B))}$.

Let $x=x^{\ast} \in A \otimes B$ be of norm at most $\frac{1}{\min(rank(A), rank(B))}$. We will show that $\mathds{1}_{A} \otimes \mathds{1}_{B} - x$ is separable. As $A$ is nuclear, we can use the Horodeckis criterion (Corollary \ref{cor:horodeckicstar}) to show that $\mathds{1}_{A} \otimes \mathds{1}_{B} - x$ is separable if and only if it remains positive after applying any unital positive map from $B$ to $A$. So we look at $\mathds{1}_{A} \otimes \mathds{1}_{B} - (Id \otimes \varphi)(x)$. As $\eta(A,B) = \min(rank(A), rank(B))$, we have $\|(Id \otimes \varphi)(x)\| \leqslant \min(rank(A),rank(B)) \|x\| = 1$, so $\mathds{1}_{A} \otimes \mathds{1}_{B} - (Id \otimes \varphi)(x) \geqslant 0$. It follows that $\mathds{1}_{A} \otimes \mathds{1}_{B} - x$ is separable, so $\gamma(A,B) \geqslant \frac{1}{\min(rank(A), rank(B))}$. To prove the other inequality, note that $\gamma(A,B) = \gamma^{\sigma}(A^{\ast\ast},B^{\ast\ast})$ and the biduals contain direct summand isomorphic to $M_{rank(A)}$ and $M_{rank(B)}$, respectively, so we can use the results for matrices (Theorem \ref{thm:gammaandetacomputation}) to conclude that $\gamma(A,B) \leqslant \frac{1}{\min(rank(A), rank(B))}$
\end{proof}

\begin{corollary}
    For any C$^*$-algebra $A$, and for any $n\in \mathbf{N}$, we have that the identity element $1_n\otimes 1_A$ in the  C$^*$-algebra $M_n(A)\cong M_n\otimes A$ has a separable neighbourhood of radius $1/n$.

    Moreover, a C$^{\ast}$-algebra $A$ is subhomogeneous if and only if for any $C^{\ast}$-algebra $B$ there is a non-trivial ball around $\mathds{1}_{A}\otimes \mathds{1}_{B}$ consisting of separable elements. In fact, it suffices to take one $C^{\ast}$-algebra $B$ of infinite rank such as the unitization of the compact operators or the UHF algebra $M_{2^{\infty}}$.
\end{corollary}
\begin{proof}
    Follows easily from the theorem as $M_n$ has rank n.

    For the second part, note that $\gamma(A,B) = \frac{1}{\min(rank(A), rank(B))}$ implies that if $rank(A) = \infty$ then we can easily find a $C^{\ast}$-algebra $B$ such that $\gamma(A,B) = 0$. On the other hand, if $rank(A)<\infty$ then $\gamma(A,B) \geqslant \frac{1}{rank(A)}$.
\end{proof}

\section{A question of Musat-R{\o}rdam in \cite{MusatRordam}}
In this section we answer a question that Musat-R{\o}rdam asked in their recent article on entanglement theory in C$^*$ algebras.

Let $\kappa(A,B) = \sup\{\|\rho\|_{max}: \rho \in S(A) \otimes^{\ast} S(B)\}$, where $S(A) \otimes^{\ast} S(B)$ is the set of functionals $\rho$ on the algebraic tensor product $A\odot B$ such that $\rho(\mathds{1}_{A}\otimes \mathds{1}_{B}) = 1$ and $\rho$ is positive on separable elements. We would like to show the following
\begin{prop}
We have $\kappa(A,B) = \min(rank(A), rank(B))$.
\end{prop}
\begin{proof}
By \cite[Proposition 3.14 (iii)]{MusatRordam} this quantity is infinite if neither $A$ nor $B$ is subhomogeneous.

So we may assume that one of them is subhomogeneous. It was already noted in \cite[Proposition 3.4, Theorem 3.15]{MusatRordam} that the parameter $\kappa(A,B)$ is related to $\eta(A,B)$, at least in the case of matrix algebras. We will show that it is still the case in our setup and we will conclude by our results on $\eta(A,B)$.

Without loss of generality, we can assume that $A$ is subhomogeneous, because $\kappa(A,B) = \kappa(B,A)$. Note that any functional on $A \otimes B$ can be viewed as a weak$^{\ast}$-continuous functional on the bidual $A^{\ast\ast}\overline{\otimes} B^{\ast\ast}$. As $A^{\ast\ast}$ is a direct sum of matrix algebras tensored with commutative algebras and separable elements split with respect to direct sums (Lemma \ref{lem:separablesplit}), we may replace $A^{\ast\ast}$ by $L^{\infty}(X) \otimes M_{d}$, where $d\leqslant rank(A)$. As $L^{\infty}(X)$ is commutative, it is easy to see that separable elements in $(L^{\infty}(X)\otimes M_d)\overline{\otimes} B^{\ast\ast}$ are the same as the ones in $M_d \otimes (L^{\infty}(X)\overline{\otimes} B^{\ast\ast})$, as in both cases they can be viewed as functions on $X$ valued in separable elements in $M_{d} \otimes B^{\ast\ast}$. In this case we have a one-to-one correspondence between functionals  $\widehat{\phi}$ positive on separable elements and positive maps $\phi$ from $L^{\infty}(X)\overline{\otimes} B^{\ast\ast}$ to $M_d$ given by $\widehat{\phi}(\sum_{i} a_{i} \otimes b_{i}) = \tr(\sum_{i} a_{i}^{T} \phi(b_i))$. Moreover, unitality of $\widehat{\phi}$ translates to $\tr(\phi(\mathds{1})) = 1$. 

We can recover $\widehat{\phi}$ using the formula $\widehat{\phi}(x\otimes y) = \tr((\sum_{i,j} e_{ij} \otimes e_{ij})(x \otimes \phi(y)))$. Now we can proceed just like in \cite[Proposition 3.4]{MusatRordam}, to show that any $\Phi \in S(M_d) \otimes^{\ast} S(L^{\infty}(X)\overline{\otimes}B^{\ast\ast})$ is of the form $\Phi = \rho \circ (Id \otimes \phi)$, where $\phi$ is a unital positive map from $L^{\infty}(X)\overline{\otimes} B^{\ast\ast} \to M_d$ and $\rho$ is a state on $M_d \otimes M_d$. It follows that $\|\Phi\| \leqslant \|\phi\|_{cb}\leqslant d\leqslant rank(A)$.

Now we prove the other inequality, namely $\kappa(A,B) \geqslant rank(A)$. By the already mentioned symmetry, we can assume that $rank(A) \leqslant rank(B)$, so $rank(A) = \min(rank(A), rank(B))$. It suffices for each $\varepsilon > 0$ to find a unital positive map $\phi:B \to A$ and a state $\rho$ on $A\otimes B$ such that $\|\rho(Id \otimes \phi)\| \geqslant  rank(A) - \varepsilon$. By Lemma \ref{lem:smithfactorization} for each $\varepsilon > 0$ there are unital completely positive maps $\rho_{A}: M_{rank(A)} \to A$ and $\sigma_{A}: A \to M_{rank(A)}$ such that $\|\sigma_A\circ \rho_A - Id\|_{cb}\leqslant \varepsilon$. As $rank(B) \geqslant rank(A)$, we can use the same argument to find unital completely positive maps $\rho_{B}: M_{rank(A)} \to B$ and $\sigma_{B}: B \to M_{rank(A)}$ such that $\|\sigma_{B}\circ \rho_{B} - Id\|_{cb}\leqslant \varepsilon$. Note that it follows that $\rho_{A}$ and $\rho_B$ are almost completely isometric. We define $\phi:= \rho_{A}\circ T \circ \sigma_{B}$, where $T: M_{rank(A)} \to M_{rank(A)}$ is the transpose map. We have 
\[
\|\phi\|_{cb} \geqslant \|\sigma_{A} \circ \phi \circ \rho_{B}\|_{cb} = \|\sigma_{A} \circ \rho_{A}\circ T \circ \sigma_{B}\circ \rho_B\|_{cb}.
\]
As the compositions $\sigma_{A} \circ \rho_{A}$ and $\sigma_{B} \circ \rho_{B}$ are close to the identity, we find that $\|\phi\|_{cb} \geqslant rank(A)(1-\varepsilon)^2$, because $\|T\|_{cb} = rank(A)$. It remains to show that the cb norm of $\phi$ can be almost attained at a self-adjoint element, because then the norm can be computed using states. It is known that the cb norm of the transpose is attained at a self-adjoint element $x \in M_{rank(A)}\otimes M_{rank(A)}$ of norm $1$. Consider $y:= (\rho_{A} \otimes \rho_B)(x)$, then $\|y\|\simeq 1$, because $\rho_A$ and $\rho_{B}$ are almost completely isometric. Note that
\[
\|(Id_{A} \otimes T)(Id_{A}\otimes \sigma_B)(y)\| \geqslant \|(Id_{A} \otimes T)(\sigma_{A}\otimes \sigma_{B})(y)\|.
\]
By definition of $y$ we obtain $\sigma_{A}\otimes \sigma_{B}(y) = (\sigma_{A}\circ \rho_{A}\otimes \sigma_{B}\circ \rho_{B})(x) \simeq x$, thus $\|(Id_{A} \otimes T)(Id_{A}\otimes \sigma_B)(y)\| \simeq rank(A)$. It follows that $\|(Id_{A} \otimes \phi)(y)\| \simeq rank(A)$ and $y$ is self-adjoint. Therefore we can find a state $\rho: A \otimes A \to \mathbb{C}$ such that $|\rho(Id \otimes \phi)(y)| \simeq rank(A)$, from which it follows that $\kappa(A,B) \geqslant rank(A) = \min(rank(A), rank(B))$. 
\end{proof}

\section{Acknowledgments}
We would like to thank the organizers of the workshop "Noncommutative Harmonic Analysis and Quantum Information" that took place in January 2026 at Oberwolfach, where the question about the separable neighbourhood of identity in C$^*$-algebras was proposed. 
This work was partially supported by the Wallenberg Centre for Quantum Technology 
(WACQT) funded by Knut and Alice Wallenberg Foundation (KAW). MW was partially supported by the National Science Center, Poland (NCN) grant no. 2021/43/D/ST1/01446.

\newcommand{\etalchar}[1]{$^{#1}$}

\end{document}